\documentclass[12pt]{amsart}

\usepackage{amsfonts,amssymb,amsmath,amsthm}
\usepackage{url}
\usepackage{bbm}
\usepackage{anysize}

\theoremstyle{plain}
\newtheorem{theorem}{Theorem}[section]
\newtheorem{lemma}[theorem]{Lemma}
\newtheorem{corollary}[theorem]{\bf Corollary}
\newtheorem{remark}[theorem]{\bf Remark}

\newtheorem{proposition}[theorem]{\bf Proposition}

\newcommand \hgt{ \mbox{ht}}
\newcommand \Aut{ \mbox{Aut}}
\newcommand \Sym{ \mbox{Sym}}
\newcommand \Id{ \mbox{Id}}

\newcommand \dom{{\rm{ dom}}}
\newcommand \rng{{\rm{ rng}}}
\newcommand \age{{\rm{ Age}}}

\newcommand \Th{{\rm{ Th}}}
\newcommand \Age{{\rm{ Age}}}

\newcommand \FMp{\KK_G}

\newcommand \BP{ \mbox{BP}}

\newcommand \zdef{{\rm{ def}}}
\newcommand \dist{{\rm{ dist}}}

\newcommand \fra{Fra\"{i}ss\'{e} }
\newcommand \uhr{\upharpoonright}

\newcommand {\cc}[1]{\mathcal{ #1 }}

\newcommand \FF{\mathcal{F}}
\newcommand \KK{\mathcal{K}}
\newcommand \AAA{\mathcal{A}}
\newcommand \BB{\mathcal{B}}
\newcommand \CC{\mathcal{C}}
\newcommand \OO{\mathcal{O}}
\newcommand \PP{\mathcal{P}}

\newcommand \NN{\mathbb{N}}

\newcommand \QQ{\mathbb{Q}}
\newcommand \Ury{\mathbb{U}}

\newcommand \ZZ{\mathbb{Z}}

\title[Remarks on weak amalgamation]{Remarks on weak amalgamation and large conjugacy classes in non-archimedean groups}

\author[M. Malicki]{Maciej Malicki}
\address{Institute of Mathematics, Polish Academy of Sciences}
\email{mamalicki@gmail.com}

\keywords{weak amalgamation, ample generics, homogenizable structures}
\subjclass[2010]{03E15, 54H11}
\thanks{Research was supported by National Agency for Academic Exchange, the Bekker Scholarship Programme PPN/BEK/2018/1/00331/U/00001.}

\begin{document}

\begin{abstract}
We study the notion of weak amalgamation in the context of diagonal conjugacy classes. Generalizing results of Kechris and Rosendal, we prove that for every countable structure $M$, Polish group $G$ of permutations of $M$, and $n \geq 1$, $G$ has a comeager $n$-diagonal conjugacy class iff the family of all $n$-tuples of $G$-extendable bijections between finitely generated substructures of $M$, has the joint embedding property and the weak amalgamation property.  We characterize limits of weak \fra classes that are not homogenizable. Finally, we investigate $1$- and $2$-diagonal conjugacy classes in groups of ball-preserving bijections of certain ordered ultrametric spaces. 
\end{abstract}

\maketitle

\section{Introduction}

Let us consider the following generalization of the notion of conjugacy class: for a group $G$, $n \geq 1$, and tuple $(g_1, \ldots, g_n) \in G^n$, the set
\[ \{ (g^{-1}g_1g, \ldots, g^{-1}g_ng) \in G^n: g \in G \} \]
is called an \emph{$n$-diagonal conjugacy class in $G$}. 
In topological groups, `large' (e.g., comeager) diagonal conjugacy classes often convey important information about the group's structure. For example, if $G$ is a Polish (i.e., separable and completely metrizable) topological group, and there exists a comeager $n$-diagonal conjugacy class in $G$ for every $n \geq 1$ (i.e., $G$ has \emph{ample generics}), the topology of $G$ is entirely determined by its algebraic structure (see \cite{KeRo}.) As a matter of fact, this is also true about various groups with a comeager $n$-diagonal conjugacy class only for $n=1$, e.g., the automorphism group $\Aut(\QQ)$ of the rational numbers.

In the context of Polish groups, most of the research on large diagonal conjugacy classes is focused on non-archimedean groups, i.e., automorphism groups of countable structures (in the model-theoretic sense.) It is known that then there is a connection between the existence of comeager diagonal conjugacy classes, and the notion of weak amalgamation. This was first established by Ivanov \cite{Iv} for $\omega$-categorical structures, and later Kechris and Rosendal \cite{KeRo} proved a general characterization to the effect that the automorphism group of the \fra limit of a \fra class $\KK$ of finite structures has a comeager $n$-diagonal conjugacy class if and only if the class $\KK_n$ of $n$-tuples of partial automorphisms of elements from $\KK$, has the joint embedding property JEP, and the weak amalgamation property WAP (see the next section for precise definitions.) They also characterized the existence of an $n$-diagonal dense conjugacy class in terms of JEP.

In fact, the Kechris--Rosendal characterization can be applied to every automorphism group $G$ of a countable structure $M$ because every such group can be realized as the automorphism group of a \fra limit $N$ that codes orbits of tuples in $M$. However, this new structure $N$ usually does not give any deeper insight into $G$ as compared with the original structure $M$, and so it is of limited help. %, and so, if $M$ is not a \fra limit to begin with, $N$ is of limited help.
In order to remove this difficulty, we generalize Kechris and Rosendal's results to all countable structures $M$, and all closed (equivalently: Polish) subgroups of $\Sym(M)$ (not necessarily automorphisms), using a variant of the Banach--Mazur game introduced by Krawczyk and Kubi\'{s} \cite{KrKu}. In Theorem \ref{th:ComeagerN}, we show that for every countable structure $M$, closed subgroup $G \leq \Sym(M)$ of permutations of $M$ (with the product topology), and $n \geq 1$, $G$ has a comeager $n$-diagonal conjugacy class if and only if the family $\KK_{G,n}$, consisting of all $n$-tuples of $G$-extendable bijections between finitely generated substructures of $M$, has JEP and WAP. Analogously (see Theorem \ref{th:DenseN}), $G$ has a dense $n$-diagonal conjugacy class if and only if $\KK_{G,n}$ has JEP.

Next, we study homogenizability of limits of weak \fra classes. Krawczyk and Kubi\'{s} \cite{KrKu} proved that hereditary classes satisfying JEP and WAP, i.e., \emph{weak \fra classes}, have a natural notion of limit that generalizes the notion of \fra limit. In light of the above discussion, it is natural to ask whether there actually exists a limit $M$ of a weak \fra class whose automorphism group cannot be viewed as the automorphism group of a \fra limit derived directly from $M$ in a finitary and constructive way. Following Covington \cite{Co} and Ahlman \cite{Ah}, we call a structure $M$ \emph{homogenizable} if there exists a finite, definable expansion $N$ of $M$ which is the limit of a \fra class (and so, in particular, $M$ and $N$ have the same automorphism group; see \cite{AtTo} for a weaker notion of homogenizability.) We show in Theorem \ref{th:WeakHomo} that a characterization of homogenizable structures proved by Ahlman \cite{Ah} turns out to be useful in this context, and we give an example of a non-homogenizable limit of a weak \fra class.

Finally, we study groups of ball-preserving bijections of ordered ultrametric spaces, objects that seem not to have been explicitly considered so far, although they have implicitly appeared in the literature devoted to structural Ramsey theory. For example, the Ramsey expansion of the class of boron trees studied by Jasinski \cite{Ja}, and Kwiatkowska and Malicki \cite{KwMa}, or Ramsey expansions of structures that can be naturally identified with Wa\.{z}ewski dendrites, studied by Kwiatkowska \cite{Kw}, can be naturally viewed as ordered ultrametric spaces with ball-preserving mappings as morphisms. In Theorems \ref{th:1-WAP} and \ref{th:2-noWAP}, we prove that groups of ball-preserving bijections of ordered ultrametric Urysohn spaces with rational distances have a comeager conjugacy class but they do not have a comeager $2$-diagonal conjugacy class. This, in particular, gives alternative, and much shorter proofs of Theorems 3.12 and 4.4 from \cite{KwMa}.
 
\section{Definitions}

A class $\KK$ of finitely generated structures in a given signature is called a \emph{\fra class} if it satisfies the following properties. It is countable up to isomorphism, it has the \emph{hereditary property} HP (for every $A\in \KK$, if $B$ is a substructure of  $A$, then $B\in \KK$),
the \emph{joint embedding property} JEP (for any $B_1,B_2\in \KK$ there exist $C\in \KK$, and embeddings $\psi_i:B_i \to C$),
and the amalgamation property AP (for any $A,B_1,B_2\in\KK$ and embeddings $\phi_i\colon A\to B_i$, $i=1,2$,
there exist $C\in\KK$ and embeddings $\psi_i\colon B_i\to C$, $i=1,2$, such that $\psi_1\circ \phi_1=\psi_2\circ \phi_2$). If, additionally, $\psi_1[B_1] \cap \psi_2[B_2]=\emptyset$ ($\psi_1[B_1] \cap \psi_2[B_2]=\psi_1 \circ \phi_1[A]$), we say that that $\KK$ has strong JEP (strong AP). And if there exists a cofinal subclass in $\KK$ with AP, we say that $\KK$ has the \emph{cofinal amalgamation property} CAP.

The class $\KK$ is called a \emph{weak \fra class} if, instead of AP, it satisfies the \emph{weak amalgamation property} WAP, i.e., for any $A \in \KK$ there is $A' \in \KK$, and an embedding $\tau:A \to A'$ such that for any $B_1,B_2\in\KK$ and embeddings $\phi_i\colon A' \to B_i$, $i=1,2$,
there exist $C\in\KK$ and embeddings $\psi_i\colon B_i\to C$, $i=1,2$, such that $\psi_1\circ \phi_1 \circ \tau =\psi_2\circ \phi_2 \circ \tau$. Any such $A'$ is called \emph{$A$-good}.

A countable structure $M$ is {\em ultrahomogeneous} if every automorphism between finitely generated substructures of $M$ can be extended to an automorphism of the whole $M$. In the case that $M$ is ultrahomogeneous, $\age(M)$, i.e., the class  of all finitely generated substructures embeddable in $M$, is a \fra class. And, by a classical theorem due to Fra\"{i}ss\'{e}, for every \fra class  $\KK$ of finitely generated structures, there is a unique up to isomorphism countable ultrahomogeneous structure $M$, called the limit of $\KK$, such that  $\KK=\age(M)$ (see \cite[Section 7.1]{Ho}.) Analogously, if $\KK$ is a weak \fra class, by results of Krawczyk and Kubi\'{s} \cite{KrKu}, there is a unique up to isomorphism countable structure $M$ satisfying a weak form of ultrahomogeneity, and such that $\KK=\age(M)$ (see \cite[Theorem 5.1]{KrKu}.) We also call this $M$ the limit of $\KK$. 

For a mapping $f$, we define $\zdef(f)=\dom(f) \cup \rng(f)$. By an orbit of $f$, we mean a maximal set $O=\{o_0, \ldots, o_n\}$ such that $f(o_i)=o_{i+1}$, $i<n$. Let $M$ be a countable structure, and let $G \leq \Sym(M)$ be a group of permutations of $M$, with the product topology. A mapping $S:A \to B$, where $A$ and $B$ are substructures of $M$, is called \emph{$G$-extendable} if it can be extended to an element of $G$. %In particular, an $\Aut(M)$-extendable mapping is simply a partial automorphism of $M$.
For a fixed $n \geq 1$, by $\KK_{G,n}$ (or by $\KK_G$, for $n=1$) we denote the family of all $n$-tuples of partial $G$-extendable mappings between finitely generated substructures of $M$. Clearly, the properties JEP, AP and WAP can be also defined in a natural way for families $\KK_{G,n}$, provided that an appropriate notion of embedding is specified. Let $\bar{S}=(S_1, \ldots, S_n)$, $\bar{T}=(T_1, \ldots, T_n)$ be tuples of $G$-extendable mappings between elements of $\KK_{G,n}$. An embedding of $\bar{S}$ into $\bar{T}$ is a $G$-extendable injection $\phi:A \to B$, where $A,B$ are substructures of $M$, such that $\zdef(S_i) \subseteq A$, $\zdef(T_i) \subseteq B$, and $\phi \circ S_i \subseteq T_i \circ \phi$, $i \leq n$; $\phi$ is an isomorphism if $\phi \circ S_i = T_i \circ \phi$, $i \leq n$. We write $\bar{S} \leq \bar{T}$ if the identity embeds $\bar{S}$ into $\bar{T}$.   

For a class $\KK_{G,n}$, by $\sigma \KK_{G,n}$, we denote the family of all chains of elements of $\KK_{G,n}$, i.e., objects of the form $\bigcup S_k$, where $S_k \in \KK_{G,n}$, and $S_k \leq S_{k+1}$, $n \in \NN$. We can define embeddings and isomorphisms between elements of $\sigma \KK_{G,n}$ as above.

\section{Weak \fra limits and ample generics}

To make the notation more transparent, in this section we usually denote elements of a class of finitely generated structures $\KK$ by letters $A,B,C, \ldots$, elements of $\KK_{G,n}$ by letters $S,T,U, \ldots$, embeddings of elements from $\KK_{G,n}$ by $\phi$, $\psi, \ldots$, and elements of $\sigma \KK_{G,n}$ by $\Phi, \Psi, \ldots$.

The following observations are straightforward.

\begin{remark}
\label{re:Emb}
Let $M$ be a countable structure, let $G \leq \Sym(M)$, and let $\Phi, \Psi \in \sigma \KK_G$.
\begin{enumerate}
\item If $\Xi$ is an embedding of $\Phi$ into $\Psi$, then $\Xi^{-1}$  is an embedding of $\Psi \upharpoonright \rng(\Phi)$ into $\Phi$. 
\item $\Phi$ and $\Psi$ are isomorphic if and only if they are conjugate by an element of $G$.
\item If $M$ is the limit of a \fra class $\KK$ of finite structures, and $G=\Aut(M)$, then $\KK_G$ is essentially the same object as $\KK_1$ in \cite{KeRo}.
\end{enumerate}
\end{remark}
 
 We say that $\Phi \in \sigma \KK_G$ is \emph{$\FMp$-universal} if every element of $\KK_G$ can be embedded into $\Phi$. And $\Phi$ is \emph{weakly $\KK_G$-injective} if it is $\FMp$-universal, and any of the conditions of the following proposition holds.
 
\begin{proposition}
\label{pr:WeakInj}
Let $M$ be a countable structure, let $G \leq \Sym(M)$, and let $\Phi \in \sigma \KK_G$. % such that $\phi$ is an automorphism of $M$.
The following conditions are equivalent:
\begin{enumerate}
\item[(a)] For every $S \leq \Phi$, $S \in \KK_G$, there exists $T \in \KK_G$ such that $S \leq T \leq \Phi$ and for every
$U \in \KK_G$ with $T \leq U$ there exists an embedding $\phi:U \rightarrow \Phi$ satisfying $\phi \uhr \dom(S) = \Id_{\dom(S)}$,
\item[(b)] for every $S \leq \Phi$, $S \in \KK_G$, there exists an isomorphism $\phi: S' \rightarrow S$, where $S' \in \FMp$, and $T \in \KK_G$ with $S' \leq T$, such that for every $U \in \KK_G$ with $T \leq U$ there exists an embedding $\psi : U \rightarrow \Phi$ extending $\phi$,
\item[(c)] for every $S \leq \Phi$, $S \in \KK_G$, and every isomorphism $\phi: S' \rightarrow S$, $S' \in \KK_G$, there exists $T \in \KK_G$ with $S' \leq T$, and such that for every $U \in \KK_G$ with $T \leq U$ there exists an embedding $\psi : U \rightarrow \Phi$ extending $\phi$.
\end{enumerate}
\end{proposition}

\begin{proof}
In order to prove (a)$\Rightarrow$(b), simply put $\phi=\Id_{\dom(S)}$.
Proving (b)$\Rightarrow$(c) is a typical diagram chasing. Fix $S \in \KK_G$, and suppose that $\phi_1:S' \rightarrow S$ and $T_1$ witness that (b) holds for $S$. Let $\Phi_1$ be an element of $G$ extending $\phi_1$. Let $\phi_2: S'' \rightarrow S$ be an isomorphism, and let $\Phi_2$ be an element of $G$ extending $\phi_2$. Then %, for $\Psi_2=\Phi_2^{-1} \circ \Phi_1 \uhr \dom(T_1)$,
$\phi_2$ and $T_2=\Xi \circ T_1 \circ \Xi^{-1}$, where $\Xi=\Phi_2^{-1} \circ \Phi_1$, also witness that (b) holds. Indeed, suppose that $U_2 \geq T_2$, $U_2 \in \KK_G$. Then $U_1=\Xi^{-1} \circ U_2 \circ \Xi$ is an element of $\KK_G$, and $T_1 \leq U_1$ because $\Xi \circ T_1 \circ \Xi^{-1}=T_2 \leq U_2$. By our assumption, there is an embedding $\psi$ of $U_1$ into $\Phi$ extending $\phi_1$. But then $\psi \circ \Xi^{-1}$ is an embedding of $U_2$ into $\Phi$ that extends $\phi_2$. The last statement holds because $\psi \circ \Xi^{-1}=\psi \circ \Phi_1^{-1} \circ \Phi_2$, both $\psi$ and $\Phi_1$ extend $\phi_1$, while $\Phi_2$ extends $\phi_2$.

To prove (c)$\Rightarrow$(a), take $\phi=\Id_{\dom(S)}$, and use (c) to find $T$ and $\psi$. Then $\psi \uhr \dom(T)$ is as required.
\end{proof}

\begin{theorem}
\label{th:WeakLim}
Let $M$ be a countable structure, and let $G \leq \Sym(M)$ be a closed subgroup such that $\KK_G$ has JEP and WAP. Then there exists a weakly $\FMp$-injective $\Phi \in G$.
\end{theorem}

\begin{proof}
As in \cite[Theorem 5.1]{KrKu}, we use the Rasiowa--Sikorski lemma, which says that for every countable partial ordering $P$, and every countable family $\cc{D}$ of cofinal subsets of $P$, there exists an increasing sequence $p_0, p_1, \ldots$ of elements of $P$ such that for every $D \in \cc{D}$ there is $n$ such that $p_n \in D$. In the present context, $P=\KK_G$ with the ordering given by inclusion. For any $m \in M$, and $S,T,U \in \KK_G$, where $T$ is $S$-good, and $T \leq U$, consider the following subsets of $\KK_G$:
\[ F_m=\{ V \in \KK_G: m \in \dom(V) \cap \rng(V) \}. \]
\[ E_{S}=\{ V \in \KK_G: S \mbox{ embeds in } V \}, \]
\[ D_{S,T,U}=\{ V \in \KK_G: \mbox{ if } T \leq V \mbox{ then } (\exists \mbox{ an embedding } \phi:U \rightarrow V) \, \phi \uhr S=\Id_{\dom(S)}  \}. \]

The sets $F_m$ are cofinal because mappings $V$ in the definition are $G$-extendable, the sets $E_S$ are cofinal by JEP, and the sets $D_{S,T,U}$ are cofinal by WAP, and because $G$-extendability of embeddings in $\KK_G$ warranties that weak amalgams over $S$ can be always chosen so that one of the embeddings of $S$ is the identity. Let $\Phi=\bigcup p_n$ be given by the Rasiowa--Sikorski lemma. Then the sets $F_m$ witness that $\Phi$ is a bijection from $M$ to $M$, and so, because $G$ is closed in $\Sym(M)$, $\Phi \in G$. The sets $E_{S}$ witness that $\Phi$ is $\FMp$-universal, and the sets $D_{S,T,U}$ witness that Proposition \ref{pr:WeakInj} (a) holds for $\Phi$, i.e., $\Phi$ is weakly $\FMp$-injective.% Finally, the sets $F_m$ witness that $\phi$ is a bijection, and so an automorphism of $M$.
\end{proof}

Now we consider the game $BM_p(G,\Phi)$ defined in \cite{KrKu}. Fix $\Phi \in \sigma \KK_G$. Both players play elements of $\KK_G$. Eve starts with some $S_0 \in \KK_G$, then Odd chooses $S_1 \in \KK_G$ such that $S_0 \leq S_1$. The players continue in this fashion, constructing a sequence $S_0 \leq S_1 \leq S_2 \leq \ldots$ of elements of $\KK_G$ whose union $\Psi=\bigcup_n S_n$ is an element of $\sigma \KK_G$. Odd wins if $\Psi$ is isomorphic to $\Phi$. 

%If $(M,\phi)$ is weakly $\KK_M$-injective, the game $BM_p(M,\phi)$ will be denoted by $BM_p(M)$. Later, we will see that $(M,\phi)$ is unique up to isomorphism, so this notation is unambigious.

\begin{theorem}
\label{th:NotInjLoses}
Let $M$ be a countable structure, let $G \leq \Sym(M)$ be a closed subgroup, and suppose that $\Phi \in \sigma \KK_G$ is not weakly $\FMp$-injective. Then Eve has a winning strategy in $BM_p(G,\Phi)$.
\end{theorem}

\begin{proof}
If $\Phi$ is not $\FMp$-universal, then Eve just starts with any element of $\FMp$ that cannot be embedded into $\Phi$. Otherwise, we use Condition (b) from Proposition \ref{pr:WeakInj}, i.e., we fix $S \leq \Phi$, $S \in \KK_G$ such that for every isomorphism $\phi:S' \rightarrow S$, $S' \in \FMp$, and every $T \in \FMp$ with $S \leq T$, there is $U \in \FMp$ with $T \leq U$ such that no embedding $\psi : U \rightarrow M$ extends $\phi$.

Eve starts with $S_0=S$. Then, at every even step $n>0$, she applies the above condition to some fixed embedding $\phi:S' \rightarrow S$, where $S' \leq S_{n-1}$, and $T=S_{n-1}$, to obtain $S_n=U$ such that no embedding of $S_n$ into $M$ extends $\phi$. By an easy bookkeeping, Eve can proceed in such a manner that for every $n$ and every embedding of $S_n$ with range containing $S$ there is $n' \geq n$ such that no embedding of $S_{n'}$ into $\Phi$ extends $\phi$. Thus, $\bigcup_n S_n$ is not isomorphic to $\Phi$. 
\end{proof}

\begin{theorem}
\label{th:InjWins}
Let $M$ be a countable structure, let $G \leq \Sym(M)$ be a closed subgroup, and suppose that $\Phi  \in \sigma \KK_G$ is weakly $\FMp$-injective. Then Odd has a winning strategy in $BM_p(G,\Phi)$, and $\Phi \in G$.
\end{theorem}

\begin{proof}
Let $\{m_{2k}\}$ be an enumeration of $M$. To begin with, let $S_0$ be the element chosen by Eve in the initial move. Since $\Phi$ is $\FMp$-universal, and $S_0$, as well as all embeddings into $\Phi$, are $G$-extendable, Odd can fix an embedding $\phi_0:S'_0\rightarrow \Phi$, with $S_0 \leq S'_0$, and with $m_0 \in \dom(S_0) \cap \rng(\phi_0)$.

Suppose now that, for some even $n$, elements $S_k$, $k \leq n$, have been selected so that, for every odd $k<n$, $S_k=T$, where $T$ is chosen by Odd using Proposition \ref{pr:WeakInj}(c) for $S=S_{i-1}$. Moreover, for every positive even $k<n$, Odd fixed an embedding $\phi_k:S'_k \rightarrow \Phi$, $S'_k \in \KK_G$, such that $S_k \leq S'_k$, $m_k \in \dom(S'_k) \cap \rng(\phi_{k})$, and $\phi_k$ extends $\phi_{k-2}$. Then Odd first fixes an embedding $\phi_{n}:S'_{n} \rightarrow \Phi$, $S'_n \in \KK_G$, such that $S_{n} \leq S'_{n}$, $m_n \in \dom(S'_n) \cap \rng(\phi_{n})$, and $\phi_{n}$ extends $\phi_{n-2}$; this is possible by the choice of $S_{n-1}$. Finally, Odd puts $S_{n+1}=T$, where $T$ is obtained by applying Proposition \ref{pr:WeakInj}(c) to $\phi=\phi_n$. In this way, regardless of what Eve does, the mapping $\Xi=\bigcup_n \phi_n$ is an embedding of $\Psi=\bigcup S_n$ into $\Phi$. Moreover, because $\dom(\Psi)=M$, and so $\Xi \circ \Psi[M]=M$, Remark \ref{re:Emb} implies that $\Xi^{-1}$ is also an embedding, and thus an isomorphism of $\Psi$ and $\Phi$. Clearly, $\Psi \in G$, and so $\Phi \in G$.
\end{proof}

\begin{theorem}
\label{th:InjJEPWAP}
Let $M$ be  a countable structure, let $G \leq \Sym(M)$ be a closed subgroup, and suppose that there exists a weakly $\FMp$-injective $\Phi \in \sigma \KK_G$. Then $\Phi$ is unique up to isomorphism, and $\FMp$ has JEP and WAP.
\end{theorem}

\begin{proof}
Suppose that $\Phi$, $\Psi \in \sigma \FMp$ are weakly $\FMp$-injective. By Theorem \ref{th:InjWins}, Odd has a winning strategy in both $BM_p(G,\Phi)$ and $BM_p(G,\Psi)$, so, in the game $BM_p(G,\Phi)$, Eve can start with an arbitrary $S_0$, and then use Odd's winning strategy for $BM_p(G,\Psi)$, while Odd uses his winning strategy for $BM_p(G,\Phi)$. Then the obtained chain $\bigcup_n S_n$ is isomorphic to both $\Phi$ and $\Psi$.

%In order to show JEP, for fixed $S,T \in \FMp$, we let Eve play $S_0=S$, and $S_2=T$, while Odd plays a winning strategy.
%Fix $T \leq (M,\phi)$ given by Proposition \ref{pr:WeakInj}(a), and fix $U,V \in \KK_M$ with $T \leq U,V$. Clearly, we can assy 

JEP directly follows from $\KK_G$-universality of $\Phi$: as any $S,T \in \KK_G$ can be embedded in $\Phi$ via some $\phi:S \rightarrow \Phi$, $\psi:T \rightarrow \Phi$, the element generated by $\Phi \uhr \rng(\phi) \cup \Phi \uhr \rng(\psi)$ witnesses that $S$ and $T$ can be jointly embedded in an element of $\KK_G$.
In order to show WAP, fix $S \in \FMp$. Without loss of generality, we can assume that $S \leq \Phi$. Find $T \leq \Phi$ with $S \leq T$, using Proposition \ref{pr:WeakInj}(a). Fix $U,V \in \KK_G$, and embeddings $\phi:T \rightarrow U$, $\psi:T \rightarrow V$. Without loss of generality, we can assume that actually $T \leq U,V$ and $\phi, \psi$ are the identity mappings on $T$. By Proposition \ref{pr:WeakInj}(a), there exist embeddings $\phi':U \rightarrow \Phi$, $\psi':V \rightarrow \Phi$ such that $\phi'$ and $\psi'$ are the identity on $S$. Thus, the element generated by $\Phi \uhr \rng(\phi') \cup \Phi \uhr \rng(\psi')$ witnesses that $U$ and $V$ can be amalgamated over $S$ in $\KK_G$.
\end{proof}

\begin{theorem}
\label{th:Comeager}
Let $M$ be a countable structure, and let $G \leq \Sym(M)$ be a closed subgroup. The following are equivalent:

\begin{enumerate}
\item $\KK_G$ has JEP and WAP,
\item there is a weakly $\KK_G$-injective $\Phi \in G$,
\item there is $\Phi \in G$ such that Odd has a winning strategy in $BM_p(G,\Phi)$,
\item $G$ has a comeager conjugacy class.
\end{enumerate}
\end{theorem}

\begin{proof}
The equivalence (1)$\Leftrightarrow$(2) follows from Theorems \ref{th:WeakLim} and \ref{th:InjJEPWAP}. The equivalence (2)$\Leftrightarrow$(3) follows from Theorems \ref{th:NotInjLoses} and \ref{th:InjWins}. To show that (3)$\Leftrightarrow$(4), observe that, by Remark \ref{re:Emb}(2), if $\Phi \in G$, we can think of $BM_p(G,\Phi)$ as the original Banach--Mazur game $G^{**}(C,G)$, played in the Polish space $G$, with the target set $C$ defined as the conjugacy class of $\Phi$.
Then the assumption that Odd has a winning strategy is equivalent to the assumption that $C$ is comeager.

\end{proof}

\begin{remark}
Note that, by Theorems \ref{th:InjWins} and \ref{th:InjJEPWAP}, in (2) and (3), we could replace the condition $\Phi \in G$ by the condition $\Phi \in \sigma \KK_G$.
\end{remark}

\begin{theorem}
\label{th:Dense}
Let $M$ be a countable structure, and let $G \leq \Sym(M)$ be a closed subgroup. The following are equivalent:

\begin{enumerate}
\item $\KK_G$ satisfies JEP,
\item there is a $\KK_G$-universal $\Phi \in G$,
\item $G$ has a dense conjugacy class.
\end{enumerate}
\end{theorem}

\begin{proof}
In order to prove (1)$\Rightarrow$(2), fix an enumeration $\{m_{2k+1}\}$ of $M$,  an enumeration $\{T_{2k}\}$ of $\KK_G$, put $S_0=T_0$, and let $S_n$, $n>0$, be an increasing sequence of elements of $\KK_G$ obtained by making sure that $m_n \in \dom(S_n) \cap \rng(S_n)$ at odd indices, and by applying JEP to $S_{n-1}$ and $T_n$ at even indices (so that the identity embeds $S_{n-1}$ into $S_n$.) Then $\Phi=\bigcup_n S_n$ is as required.

To prove that (2)$\Rightarrow$(3), fix a $\KK_G$-universal $\Phi \in G$, fix $\Psi \in G$, and $S \in \KK_G$ such that $S \leq \Psi$. Let $\Xi$ be an element of $G$ extending an embedding of $S$ into $\Phi$. Then $S \leq \Xi^{-1} \circ \Phi \circ \Xi$, so $\Xi^{-1} \circ \Phi \circ \Xi$ is in the neighborhood of $\Psi$ determined by $S$ in $G$. As $\Psi$ and $S$ were arbitrary, this shows that the conjugacy class of $\Phi$ is dense in $G$. The implication (3)$\Rightarrow$(1) is similar: if $\Phi \in G$ has a dense conjugacy class, then any $S,T \in \KK_G$ can be embedded in $\Phi$, and thus in some $\phi \subset \Phi$ with $\phi \in \KK_G$.

\end{proof}

Notice that in the proofs of the above results, we never use the fact that $\KK_G$ has HP. This means that we could also consider some suitable (cofinal) subfamily of $\KK_G$. In particular, if $M$ is the limit of a weak \fra class, and $G=\Aut(M)$, instead of $\KK_G$ we could use the family $\KK'_G$ of all partial isomorphisms $\phi: B \rightarrow C$ that can be extended to partial isomorphisms $\phi':B' \rightarrow C'$, where $B \leq B'$, $C \leq C'$ and $\dom(\phi')$ is $\dom(\phi)$-good. The benefit of considering $\KK'_G$ instead of $\KK_G$ is that the original requirement that $\phi$ is $G$-extendable cannot be verified internally, i.e., `inside' of $\phi$. As a matter fact, if $M$ is the limit of a weak \fra class, with only slight modifications of the arguments, we could obtain an analogous characterization of the existence of a dense or comeager conjugacy class, in terms of $\KK'_G$ equipped with all embeddings (not only $G$-extendable embeddings) of systems in $\KK'_G$.

Also, exactly the same proofs work if, for a given $n \geq 1$, we replace $\KK_G$ with the family $\KK_{G,n}$, and we replace the game $BM_p(G,\Phi)$ with an analagous game $BM_p(G,\Phi_1, \ldots, \Phi_n)$. Thus, we get

\begin{theorem}
\label{th:ComeagerN}
	Let $G$ be a countable structure, let $G \leq \Sym(M)$ be a closed subgroup, and let $n \geq 1$. The following are equivalent:	
	\begin{enumerate}
		\item $\KK_{G,n}$ has JEP and WAP,
		\item there are $\Phi_1, \ldots, \Phi_n \in G$ such that $(\Phi_1, \ldots, \Phi_n)$ is weakly $\KK_{G,n}$-injective,
		\item there are $\Phi_1, \ldots, \Phi_n \in G$ such that Odd has a winning strategy in $BM_p(G, \Phi_1, \ldots, \Phi_n)$,
		\item $G$ has a comeager $n$-diagonal conjugacy class.
	\end{enumerate}
\end{theorem}

\begin{theorem}
\label{th:DenseN}
	Let $M$ be a countable structure, let $G \leq \Sym(M)$ be a closed subgroup, and let $n \geq 1$. The following are equivalent:
	
	\begin{enumerate}
		\item $\KK_{G,n}$ satisfies JEP,
		\item there are $\Phi_1, \ldots, \Phi_n \in G$ such that $(\Phi_1, \ldots, \Phi_n)$ is $\KK_{G,n}$-universal,
		\item $G$ has a dense $n$-diagonal conjugacy class.
	\end{enumerate}
\end{theorem}

\begin{corollary}
Let $M$ be a countable structure. The group $\Aut(M)$ has ample generics if and only if $\KK_n=\KK_{\Aut(M),n}$ has JEP and WAP for every $n \geq 1$.
\end{corollary}

\section{Homogenizability of weak \fra classes}

In this section, we study homogenizability in the context of limits of weak \fra classes. For definitions of standard model-theoretic notions, the reader is referred to \cite{Ho}.

We say that a structure $M$ in signature $L$ is \emph{homogenizable} if there exist formulas $\phi_0(\bar{x}_0), \ldots, \phi_n(\bar{x}_n)$ such that, if we extend $L$ to a signature $L'$ obtained by adding new relational symbols $R_i$ of the same arity as $\phi_i$,  $i \leq n$, then there is an ultrahomogeneous structure $M'$ in signature $L'$ such that the reduct of $M'$ to $L$ is equal to $M$, and for each tuple $\bar{a}$ in $M'$, and $i \leq n$, we have that $R_i(\bar{a})$ holds in $M'$ if and only if $\phi_i(\bar{a})$ holds in $M'$. In other words, the relations $R_i$ are definable in $M$, and so, in particular, $\Aut(M)=\Aut(M')$.

\begin{proposition}
\label{pr:ExClosed}
Let $M$ be the limit of a weak \fra class in a finite, relational signature. Then $M$ is existentially closed.
\end{proposition}

\begin{proof}
%Fix such $X, Y \leq M$, and suppose that $N$ is a model of $Th(M)$, $M \leq N$.
	
%Observe that for every finite $X \leq M$ there exists $Y \leq M$ such that $X \leq Y$, and $Y$ is $X$-good.

Let $N$ be a model of $\Th(M)$ such that $M \subseteq N$. Fix a tuple $\bar{a}=(a_1,\ldots,a_n)$ in $M$, and an atomic formula $\phi(\bar{x},\bar{y})$ such that $\exists \bar{y} \phi(\bar{a},\bar{y})$ holds in $N$. Fix a tuple $\bar{c}=(c_1,\ldots,c_n)$ in $N$ such that $\phi(\bar{a},\bar{c})$ holds in $N$. As $M$ is the limit of a weak \fra class, $A=\{a_1,\ldots,a_n\}$ is contained in some finite $B \leq M$ as in \cite[Proposition 3.1(a)]{KrKu}. Because $N$ is a model of $\Th(M)$, we have that $\Age(M)=\Age(N)$, and $X=B \cup \{c_1,\ldots,c_n\} \in \Age(M)$. Therefore, by \cite[Proposition 3.1(a)]{KrKu}, the identity embedding of $A$ into $M$ can be extended to an embedding $f$ of $X$ into $M$, which means that $\phi(\bar{a},f[\bar{c}])$ holds in $M$. Thus, $M$ is existentially closed.
\end{proof}

\begin{corollary}
\label{co:ModCom}
Let $M$ be the limit $M$ of a weak \fra class in a finite, relational signature. If $M$ is $\omega$-categorical, then it is model-complete.
\end{corollary}

\begin{proof}
Without loss of generality, we can assume that $M$ is infinite. Let $N \subseteq N'$ be models of $\Th(M)$. Fix a tuple $\bar{a}$ in $N$, and an atomic formula $\phi(\bar{x},\bar{y})$ such that $\exists \bar{y} \phi(\bar{a},\bar{y})$ holds in $N'$. By the Skolem-L\"{o}wenheim theorem, there exists a countably infinite model $M' \leq N$ of $\Th(M)$ that contains $\bar{a}$. As $M$ is $\omega$-categorical, $M'$ is isomorphic to $M$, and, by Proposition \ref{pr:ExClosed}, it is existentially closed. Thus, $\exists \bar{y} \phi(\bar{a},\bar{y})$ holds in $M'$, and so in $N$. This shows that every model of $\Th(M)$ is existentially closed. And it is well known (see \cite[Theorem 8.3.1(b)]{Ho}) that if every model of a theory is existentially closed, then this theory is model-complete.
\end{proof}

Let $\KK$ be a class of finite structures, and let $k,m \in  \NN$. Following \cite{Ah}, we say that $\KK$ satisfies $\mbox{SEAP}_{k,m}$ (or the \emph{$(k,m)$-subextension amalgamation failure property}) if the following holds. For any $A \in \KK$, and $B, C \in \KK$, with embeddings $\phi : A \rightarrow B$, $\psi : A \rightarrow C$,
that cannot be amalgamated over $A$, there exist $A_0 \subseteq A$, $B_0 \supseteq B$ and $C_0 \supseteq C$, $A_0,B_0,C_0 \in \KK$, with $|A_0| < k$, $|B_0|-|B| < m$, $|C_0|-|C| < m$, and with emeddings $\phi_0 : A_0 \rightarrow B_0$ and $\psi_0 : A_0 \rightarrow C_0$, where
$\phi_0 = \phi \upharpoonright A_0$ and $\psi_0 = \psi \upharpoonright A_0$, that cannot be amalgamated over $A_0$. We say that $\KK$
satisfies SEAP if it satisfies $\mbox{SEAP}_{k,m}$ for some $k,m \in \NN$.

\begin{theorem}
\label{th:WeakHomo}	
The limit $M$ of a weak \fra class in a finite, relational signature is homogenizable if and only if $M$ is $\omega$-categorical and $\age(M)$ has SEAP.
\end{theorem}

\begin{proof}
	Suppose that $M$ is homogenizable. It is easy to see that it must be $\omega$-categorical, and so, by Corollary \ref{co:ModCom}, it is also model-complete. By \cite[Theorem 1.1]{Ah}, $\age(M)$ satisfies SEAP. On the other hand, if $M$ is $\omega$-categorical, by Corollary \ref{co:ModCom}, it is model-complete, and so, if $\age(M)$ satisfies SEAP, by the same theorem, $M$ is homogenizable.
\end{proof}

\subsection{An example}
It is natural to ask if there exists a limit of a weak \fra class that cannot be turned in a constructive and finitary way into a limit of a \fra class, i.e., that is not homogenizable. We sketch an example of such a class, which is a modification of a construction from \cite{KrKu2}. 

Let $L$ be a signature consisting of two binary predicates: $R$ (red) and $B$ (blue), understood as predicates denoting colored edges in a graph. By a path (tree, connected set, etc.) we mean a path (tree, connected set, etc.) in $R \cup B$, and by a monochromatic path (tree, connected set, etc.) we mean a path (tree, connected set, etc.) exclusively in $R$ or in $B$. In the case that the direction of the edges does matter, we explicitly say that a path (tree, forest, etc.) is directed. Let $\FF$ be the class of all finite structures $(A, R, B)$ in the signature $L$ with the following properties:
\begin{enumerate}
	\item the graph $(A, R \cup B)$ is an (undirected) forest, i.e., there are no (undirected) cycles in $(A, R \cup B)$, 
	\item the sets $R$ and $B$ form a partition of $R \cup B$,
	\item for every vertex $w \in A$, the set of all edges $(v,w)$ in
	$R \cup B$ is contained either in $R$ or in $B$,
	\item for every vertex $w \in A$, all directed monochromatic paths $v_1, \ldots, v_n$ ending at $w$, and such there exists $v_0 \in A$ such that $(v_0,v_1)$ has a color different than the color of edges in the path, have the same length.
\end{enumerate}

\begin{proposition}
	The class $\FF$ is a weak \fra class that does not satisfy CAP. Moreover, the limit $M$ of $\FF$ is not $\omega$-categorical, and so, in particular, not homogenizable.
\end{proposition}

\begin{proof}
	It is clear that $\FF$ has HP and strong JEP. 
	In order to see that $\FF$ has WAP, fix $A_0 \in \FF$. For $a \in A_0$, let $l(a)$ be the length of the longest, directed monochromatic path in $A_0$ that ends at $a$. We can easily extend $A_0$ to a connected $A \in \FF$ such that for every vertex $a \in A_0$, every maximal directed monochromatic path $v_1,\ldots, v_n$ ending at $a$ has length $l(a)$, and is such that there exists $v_0 \in A$ such that $(v_0,v_1)$ is an edge in $A$ of a color different from the color of edges in the path. Then for any $B,C \in \FF$ with $A \subseteq B,C$, the free amalgam $B \cup C$ (i.e., the amalgam with no new vertices or edges added) is the desired weak amalgamation of $B$ and $C$ over $A_0$. 
	
	To see that $\FF$ does not have CAP, observe that every $A \in \FF$ contains a vertex $a$ with no incoming edges. Then we can extend $A$ to an element of $\FF$ in two ways: by adding a red edge ending at $a$, or a blue edge ending at $a$. These two extensions cannot be amalgamated.
	
	It is also easy to see that the limit $M$ of $\FF$ is not $\omega$-categorical. Let $A_n \in \FF$ with fixed $a_n \in A_n$, $n \in \NN$, be elements of the form of a directed path $v_0,v_1, \ldots, v_n=a_n$ that is not monochromatic but such that $v_1, \ldots, v_n$ is monochromatic. We can assume that each $A_n$ is a subsets of $M$. Then $a_n$ witness that there are infinitely many $1$-types in $M$.
	
\end{proof}

\textbf{Question:} Does there exist a weak \fra class whose limit is $\omega$-categorical but not homogenizable?

Finally, using results from the previous section, we point out the following fact.

\begin{proposition}
	The group $\Aut(M)$ has no dense conjugacy class. 
\end{proposition}

\begin{proof}
Set $G=\Aut(M)$. Observe that $M$ is a tree, and there exists $S_0 \in \FF_G$, and $c,d \in \dom(S_0)$ connected by an edge, and not fixed by $S_0$. For example, take an element of $\FF$ the form $\{a,b,c,d,e,f\}$, where $(a,b),(c,d),(e,f)$ are red, and $(b,c), (d,e)$ are blue. Put $S_0(c)=e$, $S_0(d)=f$. Then, $\{a,b,c,d\}$ is $\{c,d\}$-good, and $S_0$ can be extended to $\{a,b,c,d\}$ by putting $S_0(a)=c$, $S_0(b)=d$. This means that $S_0 \in \FF_G$.   

Notice also that if a given $\Phi \in G$ fixes some $x \in M$, it must also fix some element of the unique path $[y,\Phi(y)]$ connecting $y$ and $\Phi(y)$, for any $y \in M$ that is not fixed by $\Phi$. Indeed, let $x$, $y$ be such elements. If $x \in [y,\Phi(y)]$, we are done. Otherwise, as $M$ is a tree, there must exist $z \in [y,\Phi(y)]$ such $[y,x] \cap [x,\Phi(y)]=[x,z]$. We leave it to the reader to verify that $\Phi(z)=z$. 
%Then $[x,y]=[x,\Phi(y)] \cup [\Phi(y),y]$. But $| [x,y] |= |[x,\Phi(y)] |$, so $[\Phi(y),y] \subseteq [x,y]$, which implies that $y \in [x,\Phi(y)]$. But this is possible only if $\Phi(y)=y$, which we assumed not to be the case.
In particular, the above shows that there is no joint embedding of $S_0$ and any $T \in \FF_G$ that fixes an element. By Theorem \ref{th:Dense}, there is no dense conjugacy class in $G$.
\end{proof}

Let us also briefly consider stabilizers of points in $\Aut(M)$. For a fixed $r \in M$, let $M_r$ be $M$ with $r$ regarded as a constant. After forgetting about colors and directions of edges, $M_r$ can be thought of as a regular, infinitely branching rooted tree $N_r$ with $r$ as the root. It was proved in \cite{Ma} that then all the corresponding classes of tuples of partial automorphisms have JEP and CAP, i.e., $\Aut(N_r)$ has ample generics. A straightforward modification of these arguments gives that all the classes $\FF_{\Aut(M_r),n}$ also have JEP and WAP, and so $\Aut(M_r)$ has ample generics.

\section{Ultrametric spaces}

In this section, we investigate groups of bijections of certain countable structures that are not groups of automorphisms. Recall that an \emph{ultrametric space} is a metric space $(X,d)$ whose metric satisfies a strong version of the triangle inequality:
\[ d(x,z) \leq \max(d(x,y),d(y,z)), \]
for any $x,y,z \in X$. Typically, ultrametric spaces are studied as metric spaces, i.e., with isometries as isomorphisms. However, one can also consider another natural kind of bijections: those that preserve balls. We will call such mappings \emph{ball-preserving bijections}, or, shortly, \emph{bp-bijections}. The group of all \emph{bp-automorphisms} of $X$, i.e., bp-bijections $\Phi:X \rightarrow X$,  will be denoted by  $\BP(X)$. A partial bp-automorphism of $X$ is a bp-bijection $p:A \rightarrow B$, where $A,B$ are finite subsets of $X$.%, which can be extended to a bp-automorphism of $X$.

Let $X$ be an ultrametric space. By a ball in $X$, we mean a set of the form
\[ B_r(x)=\{y \in X: d(x,y)<r \}, \]
$x \in X$, $r>0$, and we will assume that balls always `remember', or come equipped with, their radius. It is easy to see that for any two balls in an ultrametric space, either one is contained in the other, or they are disjoint. 

For $r>0$, by an \emph{$r$-polygon} in $X$, we mean a set $P$ such that $d(x,y)=r$ for $x \neq y \in P$. For a ball $B$ in $X$ of radius $r$, by $\PP(B)$ we denote the (pairwise disjoint) family of all balls $B'$ in $X$ of radius $r$, and such that $B'=B$ or $\dist(B', B)=r$. If $B' \in \PP(B)$ and $B' \neq B$, we say that $B'$ is \emph{adjacent} to $B$.

Let $N \in \NN \cup \{\NN\}$. By $\KK_N$, we denote the class of all finite ultrametric spaces with rational distances, and such that every $r$-polygon has size at most $N$. We will regard $\KK_N$ as a class of structures with language consisting of binary relations $d_q(x,y)$, $q \in \QQ$, defined by $d_q(x,y)$ iff $d(x,y)=q$. Then the \fra limit $\Ury_N$ of $\KK_N$ is called the \emph{rational $N$-ultrametric Urysohn space}.

Actually, we will be mostly interested in \emph{ordered ultrametric spaces}, i.e., ultrametric spaces $(X,d)$ endowed with a \emph{convex} linear ordering, i.e., a linear ordering $\prec$ satisfying
\[ x \prec y \prec z \mbox{ implies } d(x,y) \leq d(x,z). \]
%\[ x<z \mbox{ and } d(x,y)>d(x,z) \mbox{ implies } x<y, \]
%
for $x,y,z \in X$. Equivalently, a convex ordering of $X$ is an ordering $\prec$ induced by some linear ordering $\prec_B$ of balls that extends the inclusion ordering. That is, for a given such ordering $\prec_B$ of balls in $X$, we define $\prec$ by
\[ x\prec y \mbox{ iff } B_x \prec_B B_y, \]
where $B_x$, $B_y$ are the unique balls of radius equal to $d(x,y)$, and such that $x \in B_x$, $y \in B_y$. By $\KK^\prec_N$, we denote the class of finite convexly ordered ultrametric spaces with rational distances, and such that every $r$-polygon has size at most $N$. All $\KK^\prec_N$ are \fra classes, and their \fra limits $\Ury^\prec_N$ are called the \emph{ordered rational $N$-ultrametric Urysohn space}.

It is easy to verify that $\KK_N$ ($\KK^\prec_N$) with (order preserving) bp-injections as morphisms are also \fra classes. Moreover, for $N=\NN$, the ordering $\prec$ regarded as an ordering of balls in $\Ury^\prec_N$, is dense when restricted to any family $\PP(B)$. On the other hand, for finite $N$, the rigidity of $\prec$ restricted to $\PP(B)$ implies that for any balls $B, B'$ in $\Ury^\prec_N$, there is a unique $\prec$-preserving bijection between $\PP(B)$ and $\PP(B')$.

\begin{proposition}
	\label{pr:PreserveP}
	Let $(X,d_X)$, $(Y,d_Y)$ be ultrametric spaces, and let $\Phi:X \rightarrow Y$ be a bijection. The following are equivalent:
	\begin{enumerate}
		\item $\Phi$ is a bp-bijection,
		\item $\Phi$ is a bijection, and $B' \in \PP(B)$ iff $\Phi[B'] \in \PP(\Phi[B])$ for any balls $B,B'$ in $X$,
		\item $d_X(x,y)<d_X(y,z)$ iff  $d_Y(\Phi(x),\Phi(y))<d_Y(\Phi(y),\Phi(z))$ for any $x,y,z \in X$.
	\end{enumerate}	
\end{proposition}

\begin{proof}
	To prove that (1) implies (2), suppose that $\Phi$ is a bp-bijection, $B' \in \PP(B)$, and $B' \neq B$ but there is a ball $D$ such that $D \supsetneq \Phi[B]$, and $D$ is disjoint from $\Phi[B']$. Then $\Phi^{-1}[D]$ is a ball, and $\Phi^{-1}[D] \supsetneq B$, so $\Phi^{-1}[D] \supseteq B'$. But this would mean that $\Phi$ is not a bijection, a contradiction.
	
	It is obvious that (2) implies (3). And to see that (3) implies (1), fix a ball $C$ in $X$, $x \in C$, and $y \in X$ such that $C=B_r(x)$, for $r=d_X(x,y)$. It is straightforward to verify that $\Phi[C]=B_s(\Phi(x))$, where $s=d_Y(\Phi(x),\Phi(y))$.
\end{proof}

In general, it is not true that every permutation $\Phi$ of balls in an ultrametric space $X$ preserving the inclusion relation, and satisfying $B' \in \PP(B)$ iff $\Phi[B'] \in \PP(\Phi[B])$, for all balls $B, B'$ in $X$, corresponds to a bp-automorphism of $X$. However, this is true for partial bp-automorphisms as the following proposition shows. In the sequel, when studying partial bp-automorphisms, we will often regard them as appropriate bijections between finite families of balls.

\begin{corollary}
	Let $X$ be an ultrametric space. There is a correspondence between partial bp-automorphisms of $X$, and bijections $p$ between finite families of balls in $X$ such that for any $B,B' \in \dom(p)$:
	
	\begin{enumerate}
		\item there is $B'' \in \dom(p)$ such that  $B'' \neq B$, and $B'' \in \PP(B)$.
		\item $B \subseteq B'$ iff $p(B) \subseteq p(B')$
		\item $B' \in \PP(B)$ iff $p(B') \in \PP(p(B))$
	\end{enumerate}
	
	Moreover, this correspondence is categorical with bp-bijections and bijections satisfying Conditions (2) and (3) as corresponding morphisms.
\end{corollary}

\begin{proof}
	For a partial bp-automorphism $f$, let $\BB$ be the family of all balls of the form $B_r(x)$, where $x \in \dom(f)$, and $r=d(x,y)$ for some $x,y \in \dom(f)$, and let $\CC$ be defined analogously for $\rng(f)$. Then $f$ determines a bijection $p$ between $\BB$ and $\CC$. Conditions (1) and (2) are obviously satisfied by $p$, and Condition (3) follows from (2) of Proposition \ref{pr:PreserveP}. 
	
	Similarly, for a bijection $p$ between finite families of balls in $X$ satisfying (1)-(3), let $\BB$, $\CC$ be the families of balls that are $\subseteq$-minimal in $\dom(p)$, $\rng(p)$, respectively, and let $B,C \subseteq X$ be some fixed sets of representatives of $\BB$, $\CC$, respectively. Then $p$ determines a partial bp-automorphism $f:B \rightarrow C$. The `moreover' part is straightforward to verify.
\end{proof}

Fix a (partial) bp-automorphism $p$ of an ultrametric space $X$. An orbit $\OO$ of $p$ is called \emph{trivial} if it is a fixed point of $p$. It is \emph{horizontal} if $ \OO \subseteq \PP(B)$ for some $B \in \dom(p)$.  It is adjacent to an orbit $\OO'$ if there exist $O \in \OO$, $O' \in \OO'$ such that $O$ is adjacent to $O'$. 
It \emph{captures} a ball $O'$ if there is $O \in \OO$, and $\epsilon \in \{-1,1\}$ such that $O \subseteq O' \subseteq p^\epsilon[O]$ or $O' \subseteq  p^\epsilon[O] \setminus O$. And it captures an orbit $\OO'$ if there is $O' \in \OO'$ captured by $\OO$. 
We say that $p$ is \emph{complete} if $\PP(B) \subseteq \dom(p)$ whenever $B \in \dom(p)$, and $\PP(B)$ is finite. Finally, $p$ is \emph{simple} if it is complete, and there is a non-trivial $\subseteq$-monotone orbit that captures all orbits of $p$.

In the sequel, we will slightly abuse notation by writing $\BB \subseteq C$ ($\BB \subsetneq C$) also in the situation when $C$ is a subset of $X$, and $\BB$ is a family of subsets of $X$ (strictly) contained in $C$. And, for a partial bp-automorphism $p$ of $X$, $p \uhr C$ denotes the restriction of $p$ to balls contained in $C$.

The next two observations point out basic properties of various types of orbits, and relations between them.

\begin{proposition}
	\label{le:mon-or-anti}
	Let $X$ be an ultrametric space, and let $\Phi$ be a bp-automorphism of $X$. Every orbit adjacent to a $\subseteq$-monotone orbit is a $\subseteq$-antichain.
\end{proposition}

\begin{proof}
	Let $O$ be a ball whose orbit $\OO$ is adjacent to the orbit of some $O' \in \PP(O)$ such that $O' \subseteq \Phi[O']$. If $\Phi[O']=O'$, then $\OO \subseteq \PP(O)$, so $\OO$ is clearly a $\subseteq$-antichain. Otherwise, $\Phi^n[O] \in \PP(\Phi^n[O'])$, and $O \subseteq \Phi^n[O']$ for every $n>0$. Since $\Phi^n[O]$ is disjoint from $\Phi^n[O']$, it is also disjoint from $O$, for every $n>0$.
\end{proof}

\begin{proposition}
\label{pr:Capture}
Let $X$ be an ultrametric space, let $\Phi$ be a bp-automorphism of $X$, and let $\OO$ be a $\subseteq$-monotone orbit capturing a ball $O'$. Then either the orbit $\OO'$ of $O'$ is $\subseteq$-monotone or $O'$ is contained in a ball adjacent to a ball with a $\subseteq$-monotone orbit, and $\OO'$ is a $\subseteq$-antichain.
\end{proposition}

\begin{proof}
Let $O \in \OO$ witness that $\OO$ captures $O'$. Then, for some $\epsilon \in \{-1,1\}$, either $O \subseteq O' \subseteq \Phi^\epsilon[O]$ or $O' \subseteq \Phi^\epsilon[O] \setminus O$. In the latter case, as $O$ and $O'$ are disjoint, there are unique adjacent balls $B$, $B'$ such that $O \subseteq B$, and $O' \subseteq B'$. Clearly, the orbit of $B$ is $\subseteq$-monotone, and, by Proposition \ref{le:mon-or-anti}, the orbit of $B'$ is a $\subseteq$-antichain. Then $\OO'$ is also a $\subseteq$-antichain.
\end{proof}

\begin{proposition}
	\label{pr:JEP}
	For every $N \in \NN \cup \{\NN\}$, the class of partial bp-automorphisms of $\Ury^\prec_N$ has strong JEP.
\end{proposition}

\begin{proof}
	Fix $N \in \NN \cup \{\NN\}$, and partial bp-automorphisms $q_0,q_1$ of $\Ury^\prec_N$. We can regard $\zdef(q_0)$, $\zdef(q_1)$ as contained in disjoint balls $B_0$, $B_1$. Then, clearly, the union $q_0 \cup q_1$ gives a joint embedding of $q_0$ and $q_1$.% for both classes.
\end{proof}

\begin{lemma}
	\label{le:Disjoint}
	Let $N \in \NN \cup \{\NN\}$. Let $B,C$ be disjoint balls in $\Ury^\prec_N$, let $p$ be a complete partial bp-automorphism of $\Ury^\prec_N$ with $\dom(p) \subsetneq B$, $\rng(p) \subsetneq C$, and let $q_0,q_1$ be extensions of $p$ such that $\dom(q_0), \dom(q_1) \subsetneq B$, $\rng(q_0), \rng(q_1) \subsetneq C$. Then there exist partial bp-automorphisms $q'_0$, $q'_1$ such that
	\begin{enumerate}
		\item $\dom(q'_0), \dom(q'_1) \subsetneq B$, $\rng(q'_0), \rng(q'_1) \subsetneq C$,
		%\item $\dom(q'_0) \cap \dom(q'_1)=\dom(p)$, $\rng(q'_0) \cap \rng(q'_1)=\rng(p)$,
		\item $q'_0 \cup q'_1$ amalgamates $q_0$ and $q_1$ over $p$.
		
	\end{enumerate}
\end{lemma}

\begin{proof}
	We prove the lemma by induction on the well-founded ordering $\subseteq$ of possible $\dom(p)$, i.e., of the family of all finite families of (ordered) balls.
	
	% the sum of the domains of $p$, $q_0\setminus p$ and $q_1 \setminus p$.
	
	Suppose that $\dom(p) \subseteq \PP(B_0)$ for some ball $B_0 \subsetneq B$, and so $\rng(p) \subseteq \PP(C_0)$ for some ball $C_0 \subsetneq C$. Fix balls $B_1,C_1$ with $B_0 \subsetneq B_1 \subsetneq B$, $C_0 \subsetneq C_1 \subsetneq C$. Also, fix copies of families $\dom(q_0)$, $\dom(q_1)$, which contain $\dom(p)$, and are such that
	\begin{enumerate}
	\item in the copy of $\dom(q_0)$, every element is contained in $B_1$, while in the copy of $\dom(q_1)$, every element is contained in a ball from $\PP(B_0)$ or it contains $B_1$ or it is disjoint from $B_1$,
	\item $q_0 \upharpoonright B'=\emptyset$ or $q_1 \upharpoonright B'=\emptyset$ or $p \upharpoonright B' \neq \emptyset$ for any $B' \in \PP(B_0)$.
	\end{enumerate}
	Note that Condition (2) can be satisfied because $p$ is complete, so either $p$ is defined on $\PP(B_0)$ or $N=\NN$, and then the ordering $\prec$ is dense on $\PP(B_0)$. Do the same for $\rng(q_0)$, $\rng(q_1)$, $C_0$ and $C_1$, and let $q'_0$, $q'_1$ be copies of $q_0$, $q_1$, respectively, whose domains and ranges are the corresponding copies of the domains and ranges of $q_0$, $q_1$. Moreover, by strong JEP, we can assume that the restrictions of $q'_0$, $q'_1$ to elements contained in some ball from $\PP(B_0)$ are such that their union is a partial bp-automorphism. It is straightforward to verify that then $q'_0 \cup q'_1$ is a required amalgam.

	Suppose now that $\dom(p)$ is not contained in any family $\PP(B_0)$, and that the lemma is true for all strict restrictions of $p$. Let us consider two cases:
	
	Case 1: there exist $B_0, B_1 \in \dom(p)$ such that $B_0 \subsetneq B_1$. Then we have two subcases to consider. The first one is that every element of $(\dom(q_0) \cup \dom(q_1)) \setminus \dom(p)$ either contains $B_1$ or is disjoint from $B_1$. Then we can remove $B_0$ from $\dom(p)$, and use the inductive assumption. Otherwise, there exists a ball in $(\dom(q_0) \cup \dom(q_1)) \setminus \dom(p)$ that is contained in $B_1$. But then we can separately consider the restrictions $q_0 \upharpoonright B_1$, $q_1 \upharpoonright B_1$ and $p \upharpoonright B_1$, and the restrictions of $q_0$, $q_1$ and $p$ to the remaining balls, also using the inductive assumption.      
	
	Case 2: there exists a ball $B_0 \subsetneq B$ such that every element of $\dom(p)$ is contained in a ball from $\PP(B_0)$, and at least two adjacent balls in $\PP(B_0)$ contain an element from $\dom(p)$. Then we can first apply the inductive assumption to the restrictions $q_0 \upharpoonright B_1, q_1 \upharpoonright B_1$ and $p \upharpoonright B_1$, $B_1 \in \PP(B_0)$, and then proceed as in the base case $\dom(p) \subseteq \PP(B_0)$.
	
\end{proof}

Observe that exactly the same proof works if $p$ consists only of trivial orbits (and the balls $B,C$ are equal.) Thus, we have

\begin{lemma}
	\label{le:Trivial}
	Let $N \in \NN \cup \{\NN\}$, and let $q_0, q_1$ be extensions of a complete partial bp-automorphism $p$ of $\Ury^\prec_N$ consisting only of trivial orbits. Then there exists an amalgam of $q_0$, $q_1$ over $p$.
\end{lemma}

\begin{corollary}
	\label{co:Disjoint}
	Let $N \in \NN \cup \{\NN\}$. Let $p$ be a complete partial bp-automorphism of $\Ury^\prec_N$ with an $\subseteq$-antichain orbit $\OO=\{O_0, \ldots, O_n\}$ such that
	
	\begin{enumerate}
		\item $\zdef(p) \subseteq \bigcup \OO$,
		\item $\rng(p) \cap O_i= \dom(p) \cap O_i$ for $0<i<n$.
	\end{enumerate}
	
	Let $q_0$, $q_1$ be extensions of $p$ satisfying Conditions (1) and (2) (with $q_0$, $q_1$ substituted for $p$). Then there exists a partial bp-automorphism $r$ such that $\zdef(r) \subseteq \bigcup \OO$, and $r$ amalgamates $q_0$, $q_1$ over $p$.
	
	%\begin{enumerate}
	%\item ,
	%\item $r=q'_0 \cup q'_1$ is an amalgam of $q_0$ and $q_1$ over $p$, where all involved embeddings are the identity on $\zdef(p)$. 
	%\end{enumerate}
\end{corollary}

\begin{proof}
	We apply Lemma \ref{le:Disjoint} to $p \upharpoonright O_i$, $q_0 \upharpoonright O_i$, $q_1 \upharpoonright O_i$, $i<n$ to obtain amalgams $r_i$, $i<n$. It is not hard to see that we can assume that actually $\rng(r_i) \cap O_i= \dom(r_{i+1}) \cap O_i$ for every $0<i<n$, and therefore $r=\bigcup_i r_i$ is the required amalgam of $q_0$, $q_1$ over $p$.
\end{proof}

\begin{corollary}
	\label{co:Simple}
	Let $N \in \NN \cup \{\NN\}$, and let $q_0, q_1$ be simple extensions of a simple partial bp-automorphism $p$ of $\Ury^\prec_N$. Then there exists a simple amalgam of $q_0$, $q_1$ over $p$.
\end{corollary}

\begin{proof}
	Let $\OO$ be the unique $\subseteq$-monotone orbit of $p$ capturing all orbits of $p$, and let $q'_0$, $q'_1$, $p'$ be the restrictions of $q_0$, $q_1$, $p$, respectively, to balls with $\subseteq$-monotone orbits. Observe that strongly amalgamating $q'_0$, $q'_1$ over $p'$ is equivalent to strongly amalgamating corresponding partial automorphisms of finite linear orderings induced by the inclusion relation, and so there exists such a strong amalgam $r'$ with only $\subseteq$-monotone orbits.
	
	Clearly, we can assume that $r'$ extends $p'$. Observe that we can also assume that
	\begin{enumerate}
	\item every ball in $\zdef(q_0)$ or $\zdef(q_1)$ is contained in a ball $B' \in \PP(B)$ for some $B \in \zdef(r')$, and
	\item  $q_0 \upharpoonright B'=\emptyset$ or $q_1 \upharpoonright B'=\emptyset$ or $p \upharpoonright B' \neq \emptyset$ for any such $B'$.
	\end{enumerate}
  	The latter can be satisfied because $p$ is complete and, in the case that $N=\NN$, the ordering $\prec$ is dense on any $\PP(B')$. Therefore, by strong amalgamation of finite linear orderings, Lemma \ref{le:mon-or-anti}, and Corollary \ref{co:Disjoint}, for any $B \in \dom(r')$, we can find an amalgam $r_B$ of $q_0 \upharpoonright \PP(B)$, $q_1 \upharpoonright \PP(B)$ over $p \upharpoonright \PP(B)$ so that the union $\bigcup_{B \in \dom(r')} r_B$ is simple, and amalgamates $q_0$, $q_1$ over $p$.
\end{proof}

\begin{theorem}
	\label{th:1-WAP}
	For every $N \in \NN \cup \{\NN\}$, the family of partial bp-automorphisms of $\Ury^\prec_N$ has CAP.
\end{theorem}

\begin{proof}
	Fix $N \in \NN \cup \{\NN\}$, and a partial bp-automorphism $p$ of $\Ury^\prec_N$. By possibly extending it, we can assume that $p$ is complete, and the following condition holds: if, in some extension of $p$, an extension of an orbit of $p$ is captured by an orbit of $p$, then it is already captured by this orbit in $p$. Let $q_0$, $q_1$ be extensions of $p$. By possibly extending $q_0$, $q_1$, we can assume that the above condition holds for these mappings as well. 
	
	We proceed by induction on the number of non-trivial orbits of $p$. The base case follows from Lemma \ref{le:Trivial}. Suppose that there is a non-trivial $\subseteq$-monotone orbit $\OO$ in $p$. Denote by $q'_0$, $q'_1$, $p'$ the restrictions of $q_0$, $q_1$, $p$, respectively, to balls whose orbits are captured by $\OO$. As $q'_0$, $q'_1$, $p'$ are simple, by Corollary \ref{co:Simple}, there a exists a simple $r'$ amalgamating $q'_0$, $q'_1$ over $p'$. We can assume there is a $\subseteq$-monotone orbit $\OO'$ in $r'$ such that any ball in $\zdef(q'_0)$ or $\zdef(q'_1)$ is captured by $\OO'$.
	
	Next, let $q''_0$, $q''_1$, $p''$ be the restrictions of $q_0$, $q_1$, $p$ respectively, to the remaining balls. Let $C \subseteq D$ be the smallest, and the largest balls in $\OO'$. Observe that, by the condition initially imposed on $q_0$, $q_1$, $p$, we have that $B \subseteq C$ or $B \supseteq D$ or $B$ is disjoint from $D$, for any $B$ in $\zdef(q''_0)$ or $\zdef(q''_1)$ or $\zdef(p'')$. Thus, without loss of generality, we can assume that there exist $C' \subsetneq C$ and $D' \supsetneq D$ such that $p(C')=C'$, $p(D')=D'$, and $B \subseteq C'$ or $B \supseteq D'$ or $B$ is disjoint from $D'$, for $B$ as above. As there are strictly less non-trivial orbits in $p''$ than in $p$, we can find $r''$ amalgamating $q''_0$, $q''_1$ over $p''$ so that $B \subseteq C'$ or $B \supseteq D'$ or $B$ is disjoint from $D'$, for any $B \in \zdef(r'')$. It is straightforward to verify that $r' \cup r''$ amalgamates $q_0$, $q_1$ over $p$.
	
	Suppose now that there is a ball $O \in \dom(p)$ with a non-trivial horizontal orbit (this is possible only if $N=\NN$.) Denote by $q'_0$, $q'_1$, $p'$ the restrictions of $q_0$, $q_1$, $p$, respectively, to balls contained in $\bigcup \PP(O)$. Then we can find an amalgam $r'$ of $q'_0$, $q'_1$ over $p'$ such that $\zdef(r') \subseteq \bigcup \PP(O)$, using strong amalgamation of linear orderings, and Corollary \ref{co:Disjoint}. Let $q''_0$, $q''_1$, $p''$ be the restrictions of $q_0$, $q_1$, $p$ respectively, to balls not contained in $\bigcup \PP(O)$. Again, by inductive assumption, we can amalgamate $q''_0$, $q''_1$ over $p''$ so that any $B \in \zdef(r'')$ is either disjoint from $\bigcup \PP(O)$ or contains it. Then $r' \cup r''$ amalgamates $q_0$, $q_1$ over $p$.
	
\end{proof}

Below, for a word $v$ in the free group $F_2$ on two generators $s$, $t$, and partial bp-automorphisms $p$, $q$ of an ultrametric space $X$, we denote by $v(p,q)$ the partial bp-automorphism of $X$ obtained by substituting $p$ for $s$ and $q$ for $t$ in the word $v$, and performing the composition operation whenever it is possible.
 
\begin{theorem}
	\label{th:2-noWAP}
	For every $N \in \NN \cup \{\NN\}$, the class of pairs of partial bp-automorphisms of $\Ury^\prec_N$ does not have WAP. 
\end{theorem}

\begin{proof}
	For $N=\NN$, just recall that partial bp-automorphisms of families $\PP(B)$ are essentially just partial automorphisms of finite linear orderings, and the class of pairs of such automorphisms is known not to have WAP (see \cite{Tr}.)  
	
	Fix $N \in \NN$, and let $p,q$ be a pair of partial bp-automorphisms of $\Ury^\prec_N$ such that $p(C_0) \cap C_0=\emptyset$ for some fixed $C_0 \in \dom(p)$. Let $p'$, $q'$ be some extensions of $p$, $q$, respectively. Set $i=0$, let $A_i$ be any ball that strictly contains all balls in $\zdef(p') \cup \zdef(q')$, and let $\AAA_i=\{B \in \zdef(p') \cup \zdef(q'): B \subseteq A_i\}$. Let $B$ be the unique ball such that the radius of $B$ is equal to $\max \{\dist(A,B) : A,B \in \AAA_i \}$, $C_0 \subseteq B$, and every ball from $\AAA_i$ is contained in some $B' \in \PP(B)$ (note that there are at least two distinct such balls $B'$.) We put $B_i=B$, and consider the following two cases:
	
	Case 1:  $r^k(C_0) \subseteq B_i$ for every extension $r$ of $p'$ or $q'$, and $k \in \ZZ$,
	
	Case 2: $r^k(C_0) \not \subseteq B_i$ for some extension $r$ of $p'$ or $q'$, and $k \in \ZZ$.
	
	As long as Case 1 holds, we continue constructing $A_i, B_i$ by putting $A_{i+1}=B_i$, and defining $\AAA_{i+1}$, $B_{i+1}$ as above. Because $\AAA_0$ is finite, $|\AAA_i|>|\AAA_{i+1}|$, and $p(C_0) \cap C_0=\emptyset$, there must be $i_0$ such that Case 2 holds for $B_{i_0}$.
	We will show by induction on the length $n$ of sequences constructed as above that if there exists a word $v \in F_2$ and extensions $p''$, $q''$ of $p'$, $q'$ such that $v(p'',q'')(C_0) \not \subseteq B_n$, then there exist extensions $p''_0$, $p''_1$ of $p'$, and extensions $q''_0$, $q''_1$ of $q'$ such that $(p''_0,q''_0)$, $(p''_1,q''_1)$ cannot be amalgamated over $(p,q)$.
	
	For $n=0$, suppose, without loss of generality, that $r$ is an extension of $p'$. Observe that then, actually, there exists $k$, and balls $D,D'$ such that $D' \in \PP(D)$, $D_0=r^{k}(C_0) \subseteq D$, and $\AAA_0 \subseteq D'$. Clearly, we can find two extensions $q''_0$ and $q''_1$ of $q'$ so that $q''_0(D_0), q''_1(D_0) \subseteq D$, $D_0 \preceq q''_0(D_0)$ and $D_0 \succ q''_1(D_0)$. Then $(r,q''_0)$, $(r,q''_1)$ cannot be amalgamated over $(p,q)$. 
	
	Suppose now that the claim is true for all sequences of length $n$, and consider a sequence of length $n+1$. As before, we can assume that $r$ is an extension of $p'$, and fix $k \in \ZZ$, $D,D' \subseteq A_{n+1}$ such that $D' \in \PP(D)$, $D_0=r^{k}(C_0) \subseteq D$, and $\AAA_{n+1} \subseteq D'$. We have two cases to consider:
	
	Case I:  $D_0 \prec B,q'(B)$ or $D_0 \succ B,q'(B)$ for every $B \in \dom(q')$. Then, as before, we can find two extensions $q''_0$ and $q''_1$ of $q'$ so that $(r,q''_0)$, $(r,q''_1)$ are as required.
	
	Case II: $B \prec D_0 \prec q'(B)$ or $q'(B) \prec D_0 \prec B$ for some $B \in \dom(q')$. We consider only the first possibility, the other one is completely symmetric. Fix such $B$, an extension $q''$ of $q'$ such that $D_0 \in \dom(q'')$, find the largest $j \leq n+1$ such that $B \subseteq A_j$, and observe that actually $j \leq n$, and $q'(B) \subseteq A_{j+1}$. Since $D_0 \prec q'(B)$ we have that $(q'')^{-1}(D_0) \prec B$. But if $(q'')^{-1}(D_0) \subseteq A_{n+1}$, then $B \prec D_0$ would imply that $B \prec (q'')^{-1}(D_0)$, a contradiction. Thus, $(q'')^{-1}(r^k(C_0)) \not \subseteq A_{n+1}=B_n$, and we can apply the inductive assumption.
\end{proof}

\begin{corollary}
\label{co:Ultra}
	For every $N \in \NN \cup \{\NN\}$, the group $\BP(\Ury^\prec_N)$ has a comeager conjugacy class but it has no comeager $2$-diagonal conjugacy class.
\end{corollary}

In particular, we can recover Theorems 3.12 and 4.4 from \cite{KwMa}.
Recall that a boron tree structure $B$ is formed from leaves of a connected, acyclic graph $G$ all of whose vertices have order $1$ or $3$, together with a quaternary relation $R$ defined by the following condition: $R(a,b,c,d)$ iff the unique paths connecting $a$ with $b$, and $c$ with $d$, are disjoint. Given a boron tree structure $G$, an ordered boron tree structure $C$ is defined as follows. First, we choose two vertices $a,b \in G$ that are connected by an edge. Next, we turn $G$ into a binary tree $T$ with root $r$, by adding a new vertex $r$ to $G$ and new edges $\{a,r\}$, $\{r,b\}$. Finally we introduce two new relations on $B$: a linear ordering $\prec$ defined by some lexicographical ordering of $T$, and a ternary relation $S$ defined by:
\[ S(a,b,c) \mbox{ iff } a\prec b \prec c \mbox{ and } \hgt_T(a \wedge b)>\hgt_T(b \wedge c), \]
where $a,b,c \in B$, $a \wedge b$ is the meet of $a$ and $b$ in $T$, and $\hgt_T$ is the height function on $T$.

Actually, $R$ can be defined only in terms of $\prec$ and $S$, so by an ordered boron tree structure we will mean triples $(C,S_C,\prec_C)$ as above. Indeed, it is easy to see that if $R(a,b,c,d)$ holds then we can rearrange $a,b,c,d$ so that either $a \prec b\prec c\prec d$ or $a\prec c \prec d \prec b$. Then $R(a,b,c,d)$ holds iff (1) $S(a,b,c) \, \& \, \neg S(a,c,d)$ or (2) $S(a,c,b)$ or (3) $S(a,c,b) \, \& \, \neg S(a,c,d)$ or (4) $\neg S(a,b,c) \, \& \, S(c,d,b)$.

\begin{proposition}
	The class of ordered boron tree structures with embeddings as morphisms, and the class of finite ordered $2$-ultrametric spaces with bp-embeddings as morpshims are equivalent. In particular, the automorphism group of the universal ordered boron tree has a comeager conjugacy class but it has no comeager $2$-diagonal conjugacy class.
\end{proposition}

\begin{proof}
	Let $(C,S_C,\prec_C)$ be an ordered boron tree structure built out of a tree $T$. Then $T$ naturally gives rise to an ultrametric $d_C$ on $C$. Note that $(C,d_C,\prec_C)$ is an ordered ultrametric space.  It is easy to verify that, for any $a,b,c \in C$,
	\[ S_C(a,b,c) \mbox{ iff } a\prec_C b \prec_C c \mbox{ and } d_C(a,b)<d_C(b,c), \]
	so, by Proposition \ref{pr:PreserveP}, for any mapping $p:C \rightarrow D$, $p$ is an embedding of an ordered boron structure $(C,S_C,\prec_C)$ into an ordered boron structure $(D,S_D,\prec_D)$ iff $p$ is a bp-embedding of $(C,d_C, \prec_C)$ into $(D,d_D, \prec_D)$. 
	
	Analogously, any ordered $2$-ultrametric space $(C,d_C)$, gives rise to an ordered boron tree structure $(C,S^D)$ built out of the tree determined by balls in $C$.
\end{proof}

One can also consider the following generalization of $N$-ultrametric spaces. Fix $P \subseteq \{2,3, \ldots, \NN \}$. A $P$-ultrametric space is an ultrametric space $X$ together with a structure $(\BB, \{K_p\}_{p \in P})$, where $\BB$ is the family of all balls in $X$, and each $K_p$ is a unary predicate. Moreover, we require that for for every $B \in \BB$ there is a unique $p \in P$ such that $K_p(B')$ for every $B' \in \PP(B)$, and $|\PP(B)| \leq p$. We will say that $X$ is \emph{thick} if $|\PP(B)|=p$ for every finite $p \in P$, and every ball $B$ such that $K_p(B)$ holds.

In \cite{Kw}, the author studies the so called generalized Wa\.{z}ewski dendrites. Every such dendrite can be identified with the \fra limit of a class of structures that are similar to boron tree structures. For a fixed $P \subseteq \{2,3, \ldots, \NN \}$, let $\mathcal{T}_P$ be the class of finite structures $(T,R,\{K_p\}_{p \in P})$, where $T$ is a connected, acyclic graph, $R$ is the quaternary relation defined exactly as for boron tree structures, and $K_p$, $p \in P$, are unary relations such that for every $t \in T$ there is a unique $p \in P$ such that $K_p(t)$ holds, and the degree of $t$ is at most $p$. We will say that $(T,R,\{K_p\}_{p \in P})$ is \emph{thick} if for every finite $p \in P$ and $t \in T$ such that $K_p(t)$ holds, the degree of $t$ is exactly $p$.

A generalized ordered Wa\.{z}ewski dendrite can be defined as the \fra limit of the class $\mathcal{T}^\prec_P$ of expansions of elements of $\mathcal{T}_P$ in a language with a binary relation $\prec$, a family of binary relations $G_i$, $i< \max{P}$, and a ternary relation $C$. To be more specific, for a given $(T,R_T) \in \mathcal{T}_P$, we fix a thick extension $(T',R_{T'})$ of $(T,R_T)$, a root $r$ in $T'$, and a lexicographical ordering $\prec_{T'}$ of $T'$ regarded as a rooted tree. Then we define $\prec_T= \prec_{T'} \upharpoonright T$, $C$ by the condition $C(a,b,c)$ iff $R(a,b,c,r)$, and $G_i(a,b)$ by the position of the unique immediate successor of $a$ that lies between $a$ and $b$ in $T'$: $G_i(a,b)$ if $K_p(a)$ for a finite $p$, and there is an immediate successor $a' \in T'$ of $a$ such that $a \leq_{T'} a' \leq_{T'} b$, where $\leq_{T'}$ is the tree partial ordering on $T'$, and $a'$ is the $i$-th element with regard to $\prec_{T'}$ in the family of all immediate successors of $a$.

Note first that $C$ is interdefinable with the ternary relation $S$ defined for boron tree structures:
\[ S(a,b,c) \mbox{ iff } a\prec b \prec c \mbox{ and } C(a,b,c), \]
\[ C(a,b,c) \mbox{ iff } (a\prec b \prec c \mbox{ and } S(a,b,c)) \mbox{ or } (c \prec a \prec b \mbox{ and } \neg S(c,a,b)).  \]

Obviously, in the case that a structure is thick, all its relations $G_i$ are determined by the ordering $\prec$, and so they can be neglected. This means that, we can identify thick trees from $\mathcal{T}^\prec_P$ with finite and thick ordered $P$-ultrametric spaces with bp-embeddings as morphisms. Note that the relations $K_p$ can be transferred as well: $K_p(t)$ iff $K_p(B')$ for the balls $B'$ in the unique $\PP(B)$ corresponding to the family of all immediate successors of $t$. Since both of these subclasses are cofinal in their corresponding classes, and proofs of Theorems \ref{th:1-WAP} and \ref{th:2-noWAP} transfer verbatim to $P$-ultrametric spaces, we get that

\begin{proposition}
	For every $P \subseteq \{2,3, \ldots, \NN \}$, the class of partial automorphisms of elements of $\mathcal{T}^\prec_P$ has CAP, and the class of pairs of partial automorphisms of elements of $\mathcal{T}^\prec_P$ does not have WAP. In particular, the automorphism group of every generalized ordered Wa\.{z}ewski dendrite has a comeager conjugacy class but it has no comeager $2$-diagonal conjugacy class.
\end{proposition}

Finally, we point out that Theorems \ref{th:1-WAP} and \ref{th:2-noWAP} can be also used in the context of Polish ultrametric spaces and their bp-automorphism groups, equipped with the pointwise convergence topology.

\begin{theorem}
	Let $X$ be an (ordered) ultrahomogeneous Polish ultrametric space. For every $n \in \NN$, the group $BP(X)$ has a comeager $n$-diagonal conjugacy class if and only if the (countable) family of all $n$-tuples of partial bp-automorphisms of $X$ has JEP and WAP. In particular, for every ordered ultrahomogeneous Polish ultrametric space $X$, the group $BP(X)$ has a comeager conjugacy class but it has no comeager $2$-diagonal conjugacy class. 
\end{theorem}

\begin{proof}
	
	Fix $n \in \NN$. Because the family of finite subspaces of $X$ is countable up to isometric isomorphism, the family $\KK$ of $n$-tuples of partial bp-automorphisms of $X$ is also countable up to bp-isomorphism. Suppose that $\KK$ has JEP and WAP. Applying the construction from the proof of Theorem \ref{th:WeakLim}, we can show that the set of weakly $\KK$-injective bp-automorphisms of $X$ is a dense $G_\delta$ subset of $\BP(X)$. It is easy to prove that any two weakly $\KK$-injective $\Phi, \Psi \in \BP(X)$ are conjugate. Indeed, fix weakly $\KK$-injective $\Phi, \Psi \in \BP(X)$, and a countable, dense $X_0 \subseteq X$ that is invariant under the action of both $\Phi$ and $\Psi$. Let $G_0=\BP(X_0)$. Then, $\Phi_0=\Phi \upharpoonright X_0$ and $\Psi_0=\Psi \upharpoonright X_0$ are also $\KK$-weakly injective, and so, by Theorem \ref{th:InjJEPWAP}, they are conjugate by an element $\Xi_0$ of $G_0$. Now, $\Xi_0$ uniquely extends to $\Xi \in \BP(X)$, and $\Xi$ witnesses that $\Phi$ and $\Psi$ are conjugate in $\BP(X)$.
	
	Similarly, it is a straightforward observation that if there exists a comeager $n$-diagonal conjugacy class in $\BP(X)$, then $\KK$ has JEP and WAP. The last statement of the theorem follows then from Theorems \ref{th:1-WAP} and \ref{th:2-noWAP}.
\end{proof}

\end{document}